\documentclass[11pt, onecolumn]{IEEEtran}
\usepackage{epsf}
\usepackage{bm}
\usepackage{latexsym}
\usepackage{amsmath}
\usepackage{amssymb}
\usepackage{amsthm}
\usepackage{graphicx}
\usepackage{subfigure}
\usepackage[usenames,dvipsnames,svgnames,table]{xcolor}
\definecolor{second_color}{rgb}{0,0,.3}
\usepackage{mathtools}
\RequirePackage{amsopn}
\RequirePackage{amsfonts}
\usepackage [english]{babel}
\usepackage [autostyle, english = american]{csquotes}
\usepackage{cite}
\MakeOuterQuote{"}

\mathtoolsset{showonlyrefs}

\title{\LARGE \bf EMPIRICAL CENTROID FICTITIOUS PLAY: AN APPROACH FOR DISTRIBUTED LEARNING IN MULTI-AGENT GAMES}
\author{BRIAN SWENSON$^{\dagger*}$, SOUMMYA KAR$^\dagger$ AND JO\~{A}O XAVIER$^\star$\thanks{The work was partially supported by the FCT project [PEst-OE/EEI/LA0009/2011] through the Carnegie Mellon/Portugal Program managed by ICTI from FCT and by FCT Grant FCT PTDC/EMS-CRO/2042/2012. \newline$^\dagger$Department of Electrical and Computer Engineering, Carnegie Mellon University, Pittsburgh, PA 15213, USA (brianswe@andrew.cmu.edu and soummyak@andrew.cmu.edu).\newline $^\star$Institute for Systems and Robotics (ISR), Instituto Superior Tecnico (IST), Technical University of Lisbon, Portugal (jxavier@isr.ist.utl.pt).}}

\newcommand{\ones}{{\bf{1}}}

\newtheorem{theorem}{Theorem}
\newtheorem{lemma}{Lemma}

\newtheorem{assumption}{A.}

\newtheorem*{remark}{Remark}

\begin{document}

\maketitle
\thispagestyle{empty}

\begin{abstract}
The paper is concerned with distributed learning in large-scale games. The well-known fictitious play (FP) algorithm is addressed, which, despite theoretical convergence results, might be impractical to implement in large-scale settings due to intense computation and communication requirements.  An adaptation of the FP algorithm, designated as the empirical centroid fictitious play (ECFP), is presented. In ECFP players respond to the centroid of all players' actions rather than track and respond to the individual actions of every player. Convergence of the ECFP algorithm in terms of average empirical frequency (a notion made precise in the paper) to a subset of the Nash equilibria is proven under the assumption that the game is a potential game with permutation invariant potential function. A more general formulation of ECFP is then given (which subsumes FP as a special case) and convergence results are given for the class of potential games.
Furthermore, a distributed formulation of the ECFP algorithm is presented, in which, players endowed with a (possibly sparse) preassigned communication graph, engage in local, non-strategic information exchange to eventually agree on a common equilibrium. Convergence results are proven for the distributed ECFP algorithm.
\end{abstract}


\section{Introduction}
The theory of learning in games is concerned with the study of dynamical systems induced by repeated play of a normal-form game. The general question of interest is, can player behavior rules be assigned that ensure players eventually learn a Nash equilibrium (NE) strategy? In this paper, we address the more focused question: can player behavior rules be assigned that ensure players learn a NE strategy \emph{and are practical in games with a large number of players}?

In particular, we focus on the well-known fictitious play (FP) algorithm. While FP does not converge\footnote{In reference to FP or one of its variants, we use the term convergence to mean that the empirical frequency distribution of a FP process asymptotically converges to the set of Nash equilibria, a notion to be made precise in Section \ref{sec2}. } in all games \cite{shapley1964,jordan1993,foster1998nonconvergence}, it has been proven that FP converges in games with an arbitrarily large number of players under the assumption that the game is a potential game \cite{Mond01,Mond96}. This theoretically promising result suggests that FP might be an ideal algorithm for some large-scale settings; however, the prohibitively demanding communication and computational requirements of the algorithm make any large-scale implementation highly challenging, if not impractical. In particular, it is to be observed that FP, in its classical form, may not be practical for implementation in large-scale games because of the following problems:
\begin{enumerate}
    \item demanding communication requirements,
    \item demanding memory requirements,
    \item high computational complexity.
\end{enumerate}
We present an adaptation of FP which mitigates these problems, and by these criteria, is well suited to large-scale games.

The traditional FP implementation and its variants assume that each player has instantaneous access to the time-varying action histories of all other players. Moreover, such information access or gathering is assumed to be free, which makes these implementations infeasible in practical large network scenarios. In fact, in such large networks, each player may directly observe the actions of a few neighboring players only, whereas the actions of the others need to be communicated \emph{efficiently} through an underlying communication infrastructure. In order to explicitly account for the communication costs that are to be incurred in the information gathering/dissemination process, and thereby address problem (1) of demanding communication requirements of traditional FP, we assume a problem framework where players are permitted to communicate via a preassigned (sparse but connected) communication graph $G = (V,E)$. In this framework, a node in the graph represents a player, and an edge in the graph signifies the ability for players to exchange information.
Furthermore, in our setup we assume that communication and the repeated game play evolve according to the same discrete time clock and that players are permitted only one round of communication per round of game play.

We refer to the classical setting, where players are assumed to have instantaneous access to the action histories of other players, as a \emph{centralized-information} setting. We refer to the novel setting presented in this paper, where players have a limited ability to track the actions of others but may communicate with a local subset of neighboring players via some underlying communication structure, as a \emph{distributed-information} setting. We also note, in this context, in the majority of game theory literature, a graph structure denotes the ability of a player to observe the actions of a neighbor \cite{gale2003bayesian,eksinlearning,Shamma03,rosenberg2009informational}, not to exchange information, as in our approach.

In order to address problems (2) and (3), we propose a new variant of FP wherein players respond to the centroid of the marginal empirical distributions (the centroid distribution) rather than track and respond to the entire tuple of independent marginal distributions.  We call this algorithm empirical centroid fictitious play (ECFP).  The advantages of this approach are a mitigation of the FP complexity problem (see Section \ref{sec_BR_computation}) enabled by a computationally simplified best response rule (thereby mitigating (3)), and a mitigation of the FP memory problem by requiring players to track only a single distribution (the centroid distribution) which is invariant to the number of players in the game (thereby mitigating (2)). ECFP may be implemented in both the classical centralized-information setting or the distributed-information setting that we introduce in this paper.

Our main contributions are twofold: \newline
\textbf{Main Contribution 1: We present empirical centroid fictitious play (ECFP)---an adaptation of FP that mitigates the memory and computational demands of classical FP.} ECFP addresses problems (2)--(3) associated with FP in large-scale games. We show that ECFP converges in terms of {\it{average}} empirical distribution to a subset of the mixed strategy Nash equilibria, which we call the consensus equilibria. Convergence results are proven for games with identical permutation-invariant utility functions and can be extended to the larger class of games known as potential games \cite{Mond96}, with the restriction that the potential function be permutation invariant.

In section \ref{sec6} a generalized formulation of ECFP is given. In the generalized formulation the set of players is partitioned into classes, and players track multiple centroids---one centroid for each class. The formulation allows some of the initial assumptions regarding ECFP to be relaxed. Moreover, classical FP occurs as a special case of ECFP in this formulation.

A significant issue with many learning algorithms in large-scale games is the rate of convergence in terms of the number of players. While we do not present a formal analysis of the convergence rate of ECFP in this paper, we note that experimental results and illustrative case studies (see Section VIII) suggest that the convergence rate of ECFP (in its basic formulation) tends to be invariant to the number of players.

\textbf{Main Contribution 2: We present distributed ECFP---an implementation of ECFP in which agent policy update depends only on local neighborhood information exchange.} The presented algorithm is an implementation of ECFP within our distributed-information framework. It addresses all three problems (1)--(3) associated with FP in large games. We prove convergence of the algorithm to the set of consensus equilibria.  This convergence guarantees that each agent obtains an accurate estimate of the limiting equilibrium strategy.

\subsection{Related Work}
An overview of the subject of learning in games is found in \cite{fudenberg1998theory}. Many large-scale learning algorithms exist that are not based on FP, including no-regret algorithms \cite{jafari2001no,marden2007regret}, aspiration learning \cite{chasparis2010aspiration}, and other model-free approaches \cite{marden-young-2011,pradelski2012learning,marden-shamma-08}. These learning algorithms tend to be fundamentally different than FP in that they do not track past actions of other players.

Variants of FP have been proposed for two player games \cite{fudenberg1995consistency,leslie2006generalised,washburn2001new,fudenberg1998theory} and are generally aimed at improving various aspects of the two player algorithm (i.e., faster convergence, convergence in specific games, etc.).

Sampled FP \cite{Lambert01} addresses the problem of computational complexity of FP in large-scale games by using a Monte Carlo method to estimate the best response.  Although computationally simple in the initial steps of the algorithm, the number of samples required to ensure convergence grows without bound.

Dynamic FP \cite{Shamma03} applies principles of dynamic feedback from control theory to improve the convergence properties of a continuous-time version of FP. The algorithm is shown to be stable around some Nash equilibria where traditional FP is unstable. While the results generalize to multi-player games, there is no mitigation of the information gathering problem.  In \cite{Arslan04}, a similar algorithm utilizing only payoff-based dynamics is presented. Similar stability results are shown when the class of games is restricted to games with a pairwise utility structure.

Joint strategy FP \cite{marden06} is shown to converge for generalized ordinal potential games. Players track the utility each of their actions would have generated in the previous round, and then use a simple recursion to update the predicted utility for each action in the subsequent round. Actions are chosen by maximizing the predicted utility. In joint strategy FP, the information tracking problem is mitigated by requiring agents to track only the information germane to the computation of the predicted utility for actions of interest. No information gathering scheme is explicitly defined; players are assumed to have full access to the necessary information at all times. In distributed ECFP, proposed in this paper, the information gathering scheme is explicitly defined via a preassigned (but arbitrary) communication graph, and convergence results are demonstrated when inter-agent communication is restricted to local neighborhoods conformant to the graph.

Na and Marden \cite{na-marden-2011b} present a systematic methodology for designing local-agent utility functions such that the NE of the designed game coincide with optimizers of a prespecified objective function; moreover, the agent utility functions may be designed to achieve a desired degree of `locality', in the sense that the designed utility functions depend only on information from a set of local neighboring agents. The desired degree of locality is achieved by augmenting agents' action space with a set of state space variables consisting of a value and an estimate of the value of each opponent. Although the context is quite different from distributed ECFP (designing the utility structure vs. designing learning dynamics), both works seek to address the communication problem in large games by restricting information exchange to a local neighborhood of each player. However, in \cite{na-marden-2011b} agents estimate the `value' held by each of the other $n-1$ players, hence the memory size of the messages that must be passed grows linearly with the size of the game, as opposed to distributed ECFP in which the memory size of messages is invariant to $n$.

In payoff based approaches (e.g., \cite{marden2009payoff}, \cite{leslie2006generalised}) it is assumed that players measure the instantaneous payoff information and base future action choices off this information alone. Payoff based approaches apply when instantaneous payoff information is available, and are generally very effective at mitigating communication, memory, and complexity requirements. However, there are scenarios in which, even if the utility structure is available, the instantaneous payoffs are not---in such cases ECFP type approaches are applicable.

The work \cite{koshal2012gossip}, studies the problem of learning NE in a continuous kernel aggregative game by sharing information through an overlaid communication graph. The notion of using a communication graph to estimate the aggregate behavior is similar to distributed ECFP where players use the graph to estimate the empirical centroid. However, in \cite{koshal2012gossip} the underlying game is fundamentally different (continuous kernel), and the communication scheme is based on asynchronous gossip. In this context, see also \cite{Swenson-MFP-Asilomar-2012}, our preliminary work on ECFP, which introduces the concept of graph-theoretic aggregation of information in repeated play settings of the type considered in this paper.

The remainder of the paper is organized as follows:  Section \ref{sec2} sets up notation to be used in the subsequent development and introduces the notion of consensus equilibria. In Section \ref{sec2A} the classical FP algorithm is reviewed. Section \ref{sec3} introduces ECFP as a low-information-overhead, repeated-play alternative to FP for learning consensus equilibria in multi-agent games.  Section \ref{sec4} presents the distributed-information learning framework. A fully distributed implementation of the proposed ECFP, the distributed ECFP, in multi-agent scenarios in which agent information dynamics is restricted to communication over a preassigned sparse communication network is presented and analyzed in section \ref{sec5}.  In section \ref{sec6} we discuss the generalized formulation of ECFP, and present generalized convergence results.  In section \ref{sec7} we demonstrate an application of distributed ECFP in a cognitive radio scenario with a view to illustrating the analytical concepts developed in the paper.  Finally, section \ref{sec8} concludes the paper.

\section{Preliminaries}
\label{sec2}
\subsection{Game Theoretic Setup}
A normal form game is given by the triple $\Gamma = (N,(Y_i)_{i\in N},(u_i(\cdot))_{i\in N})$, where $N = \{1,\ldots,n\}$ represents the set of players, $Y_i$---a finite set of cardinality $m_i$---denotes the action space of player $i$ and $u_i(\cdot):\prod_{i=1}^n Y_i \rightarrow \mathbb{R}$ represents the utility function of player $i$.

In order to guarantee the existence of NE and work in an overall richer framework, we consider the mixed extension of $\Gamma$ in which players may use probabilistic strategies, and players' payoffs are extended to account for such strategies.

The set of mixed strategies for player $i$ is given by $\Delta_i = \{p \in \mathbb{R}^{m_i}:\sum_{k=1}^{m_i} p(k) = 1,~p(k) \geq 0~\forall k=1,\ldots,m_i \}$, the $m_i$-simplex.   A mixed strategy $p_i \in \Delta_i$ may be thought of as a probability distribution from which player $i$ samples to choose an action. The set of joint mixed strategies is given by $\Delta^n = \prod_{i=1}^n \Delta_i$. A joint mixed strategy is represented by the $n$-tuple $\left(p_1,p_2,\ldots,p_n \right)$, where $p_i\in \Delta_i$ represents the marginal strategy of player $i$, and it is implicity assumed that players' strategies are independent.

A pure strategy is a degenerate mixed strategy which places probability $1$ on a single action in $Y_i$. We denote the set of pure strategies by $A_i = \{e_1,e_2,\ldots e_{m_i}  \}$ where $m_i$ is the number of actions available to player $i$, and $e_j$ is the $j$th canonical vector in $\mathbb{R}^{m_i}$. The set of joint pure strategies is given by $A^n = \prod_{i=1}^n A_i$.

The mixed utility function for player $i$ is given by the function $U_i(\cdot):\Delta^n \rightarrow \mathbb{R}$, such that,
\begin{equation}
U_i(p_1,\ldots,p_n) := \sum\limits_{y \in Y} u_i(y)p_1(y_1)\ldots p_n(y_n).
\label{mixed_U}
\end{equation}
Note that $U_i(\cdot)$ may be interpreted as the expected value of $u_i(y)$ given that the players' mixed strategies are statistically independent. For convenience the notation $U_{i}(p)$ will often be written as $U_{i}(p_i,p_{-i})$, where $p_i \in \Delta_i$ is the mixed strategy for player $i$, and $p_{-i}$ indicates the joint mixed strategy for all players other than $i$.  This paper will often deal with games with identical utility functions such that $U_i(p) = U_j(p)\; \forall i,j$; in such cases we drop the subscript on the utility of player $i$ and write $U(p) = U_i(p)\;\forall i$.

The set of Nash equilibria of $\Gamma$ is given by $NE := \{p\in \Delta^n :U_{i}(p) \geq U_{i}(g_i,p_{-i})\; \forall g_i \in \Delta_i, ~\forall i \},$
and the subset of consensus equilibria\footnote{The concept of a consensus equilibrium is closely related to that of a symmetric equilibrium.  The existence of symmetric equilibrium in finite normal form games was first proven by Nash \cite{Nash51} in the same work where the concept of Nash equilibrium was originally presented.  In general, a symmetric equilibrium is a Nash equilibrium that is invariant under automorphisms of the game.  A consensus equilibrium, on the other hand, is a Nash equilibrium in which all players use an identical strategy.  In the case of a symmetric game, the two concepts coincide.} as $ \nonumber C := \{p \in NE~:~  p_1 = p_2 = \cdots = p_n \}.$ The set of $\varepsilon$-Nash equilibria is given by
\begin{equation}
NE_\varepsilon := \{p \in \Delta^{n}~:~ U_{i}(p) \geq U_{i}(g_i,p_{-i})-\varepsilon,\; \forall g_i \in \Delta_i, ~\forall i\},
\label{epsilon_eqilib_eq}
\end{equation}
and the set of $\varepsilon$-consensus equilibria as
\begin{align}
\nonumber C_{\varepsilon} := \{p\in NE_{\varepsilon}: ~p_1 = p_2 = \cdots = p_n \}.
\end{align}

The distance of a distribution $p \in \Delta^n$ from a set $S \subset \Delta^n$ is given by $d(p,S) = \inf \{ \| p - p' \| : p'\in S\}$. Throughout the paper $\|\cdot\|$ denotes the standard $\mathcal{L}_{2}$ Euclidean norm unless otherwise specified.  For $\delta > 0$ we denote the set $B_\delta(C) = \{p\in\Delta^n:p_1 = p_2= \ldots =p_n \mbox{ and } d(p,C)<\delta  \}.$

In what follows (with the exception of Section VII where we pursue generalizations), we will restrict attention to games with identical permutation-invariant utilities; formally, we assume:
\begin{assumption}
\label{ass1}
All players use the same strategy space.
\end{assumption}
\begin{assumption}
\label{ass2}
The players' utility functions are identical\footnote{Games where players have identical utility functions are referred to as `identical interests games' \cite{Mond01}. They are closely related to potential games \cite{Mond96} and have many useful applications in both engineering and economics \cite{Voor00,Rose73,na-marden-2011b,marden-connections}. } and permutation invariant. That is, for any $i,j\in N$, $u_i(y) = u_j(y)$, and $u([y']_i,[y'']_j,y_{-(i,j)}) = u([y'']_i,[y']_j,y_{-(i,j)}),$
where, for any player $k\in N$, the notation $[y']_i$ indicates the action $y'\in Y_k$ being played by player $k$, and $y_{-(i,j)}$ denotes the set of actions being played by all players other than $i$ and $j$.
\label{a2}
\end{assumption}
Note that because of assumption \textbf{A.\ref{ass1}}, permutations of the form given in \textbf{A.\ref{ass2}} are necessarily well defined. Also note that under these assumptions, the set of consensus equilibria is known to be nonempty \cite{Cheng04}.


\subsection{Repeated Play}
In a repeated-play learning algorithm, players repeatedly face off in a fixed game $\Gamma$. An algorithm designer's objective is to design the behavior rules of individual players in such a way as to ensure that the players eventually learn a Nash equilibrium of $\Gamma$ through the repeated interaction.

Let $\left\{ a_i(t) \right\}_{t=1}^\infty$ be a sequence of actions for player $i$, where $a_i(t) \in A_{i}$.  Let $\{ a(t) \}_
{t=1}^\infty$ be the associated sequence of \textit{joint} actions $a(t) = (a_1(t),\ldots,a_n(t)) \in A^n$.  Note that $a_i(t) \in \mathbb{R}^{m_i}$; when necessary, we denote the $k$th element of the vector $a(t)$ by $a(t,k)$.

Let $q_i(t)$ be the normalized histogram (empirical distribution) of the actions of player $i$ up to time $t$, i.e., $q_i(t) = \frac{1}{t} \sum_{s=1}^t a_i(s)$. Similarly, let $ q(t) = \frac{1}{t} \sum_{s=1}^t a(s)$ be the \textit{joint} empirical distribution corresponding to the joint actions of the players up to time $t$.

Under assumption \textbf{A.1} (i.e., when all players have identical action spaces), let $ \bar q(t) = (q_1(t)+q_{2}(t)+\cdots+q_{n}(t))/n.$ Note that $\bar{q}(t) \in \mathbb{R}^{m}$, where $m$ denotes the cardinality of the action spaces (assumed identical) of the individual players. We refer to $\bar q(t)$ as the (eponymous) empirical centroid distribution, or the average empirical distribution.
Let $\bar q^n(t) = \left(\bar q(t), \bar q(t),\ldots ,\bar q(t)\right) \in \Delta^n$ denote the mixed strategy where all players use the empirical average as their individual strategy.

In the exposition of ECFP and distributed ECFP, several terms will be introduced that are related to the empirical centroid. Though we do not define all these terms now, the table below may be used for reference in determining the relationship between these terms.

\begin{figure}[t]
\begin{center}
\includegraphics[scale=0.6]{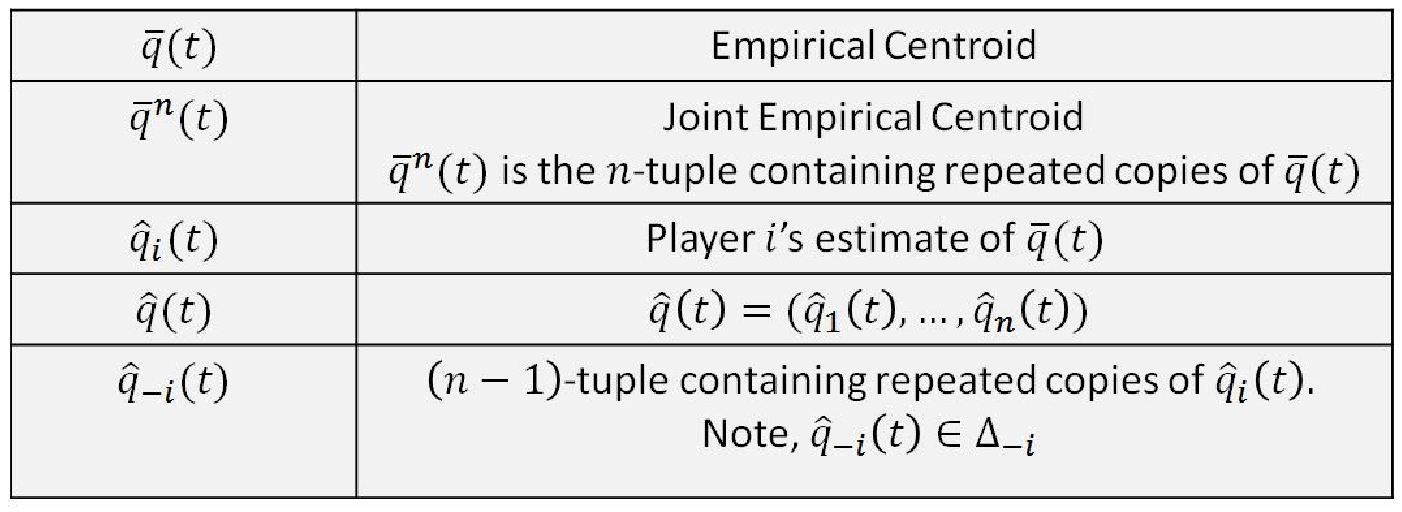}
\caption{Terms related to the empirical centroid.}
\label{table_ECFP}
\end{center}
\end{figure}
\vspace{-0.2cm}

\section{Fictitious Play}
\label{sec2A}
A fictitious play process is a sequence of actions $\{a(t)\}_{t\geq 1}$ such that, for all $i$ and $t$,\footnote{The initial action $a(1)$ may be chosen arbitrarily.}
\begin{align}
a_i(t+1) \in \arg\max \limits_{\alpha_i \in A_i} U(\alpha_i,q_{-i}(t)).
\label{FP_BR}
\end{align}
Intuitively speaking, this describes a process where each player (naively) assumes that her opponents are playing according to stationary independent strategies. Following this intuition, the player assumes that $q_{-i}(t)$ accurately represents the mixed strategy of her opponents and chooses a next-stage action in order to myopically optimize her next-stage utility.

In games where an FP process leads players to learn a NE, the equilibrium learning traditionally occurs in the sense that the empirical frequency distribution converges to the set of Nash equilibria, i.e., $d(q(t),NE)\rightarrow 0$ as $t\rightarrow \infty$. We refer to this form of learning as convergence in empirical distribution. In \cite{Mond01} it was shown that a fictitious play process converges in this sense for games satisfying \textbf{A.\ref{ass1}} -- \textbf{A.\ref{ass2}}. While theoretically promising, the result is of limited practical value due to the difficulties involved in a large-scale implementation of FP. In particular, we propose that FP, in its classical form, may not be practical for implementation in large-scale games because of the following 3 problems: (1) demanding communication requirements, (2) demanding memory requirements, and (3) high computational complexity. We discuss each in detail below.
\vspace{-1em}
\subsection{Demanding communication requirements}
\label{sec2A_1}
Implicit in \eqref{FP_BR} is the assumption that agent $i$ has instantaneous access to the action history of each opponent. We refer to this setting, where players have instantaneous access to information about the action histories of all opponents, as a \emph{centralized-information} setting. This might be impractical in a game with a large number of players.

In this paper we consider a more practical and realistic setting in which players may be incapable of directly observing the actions of opponents but are permitted to exchange information with a local subset of neighboring players in order to estimate action histories they cannot directly observe. We refer to this setting as a \emph{distributed-information} setting; the topic is discussed in section \ref{sec4}.
\vspace{-1em}
\subsection{Demanding memory requirements}
Inspection of \eqref{FP_BR} shows that each player must track the vector $q(t) = (q_1(t),\ldots,q_n(t))$. The size of this vector grows linearly in the number of players. In order to mitigate this problem, we consider ECFP, a variant of FP which uses a modified best response function. In ECFP players track only the centroid distribution, $\bar q(t)$, the memory size of which is invariant to the number of players.
\vspace{-1em}
\subsection{High computational complexity}
In order to choose a next-stage action, player $i$ must solve the optimization problem \eqref{FP_BR}. The computational complexity of computing the mixed utility given an $n$-dimensional probability density function grows exponentially with the number of players. In order to address this issue, we consider ECFP, a variant of FP that mitigates this problem by means of a modified best response rule.

\section{Empirical Centroid Fictitious Play}
\label{sec3}
The key idea of ECFP is a modification of the FP best response rule \eqref{FP_BR} which allows for mitigations in the problems of memory, and computational complexity, associated with FP.

Consider a scenario in which a player does not have the ability to track the individual actions, $a_i(t)$, of any single player. Rather, a player is only able to track the {\it{average}} action, $\bar a(t) := \frac{1}{n}\sum_{i=1}^n a_i(t)$, of the collective and therefore has access only to the {\it{average}} empirical distribution, $\bar q(t) = \frac{1}{n}\sum_{i=1}^n q_i(t)$.  An ECFP process is a sequence of actions $\{a(t)\}_{t\geq 1}$ such that for all $i$ and $t$,\footnote{The initial action, $a(1)$, may be chosen arbitrarily.}
\begin{equation}
a_i(t+1) \in \arg\max\limits_{\alpha_i \in A_i} U(\alpha_i, \bar q_{-i}(t)),
\label{v_m_eq}
\end{equation}
where $\bar q_{-i}(t) = (\bar q(t),\ldots,\bar q(t))$ is the $(n-1)$-tuple containing $(n-1)$ repeated (identical) copies of $\bar q(t)$. Intuitively speaking, this describes a process where each player (naively) assumes opponents are playing mixed strategies which are stationary, independent, and \emph{identical} --- the last assumption (identical opponent strategies) being the primary difference between ECFP and FP. The player naively assumes $\bar q(t)$ accurately represents the (identical) mixed strategy used by each opponent, and accordingly, chooses a best response to myopically optimize her next-stage utility.

In an ECFP process \eqref{v_m_eq}, the problem of demanding memory requirements is mitigated by requiring players to track only the centroid distribution, $\bar q(t)$, a vector whose memory size is invariant to the number of players in the game.

In ECFP, the mitigation in computational complexity is enabled by the introduction of the distribution $\bar q_{-i}(t)$ in the best response computation. In the joint distribution $\bar q_{-i}(t)$, all players are assumed to use independent and identical mixed strategies. Analogous to the manner in which independent and identically distributed (i.i.d.) random variables always make life simpler in statistics, the introduction of the "i.i.d." joint distribution $\bar q_{-i}(t)$ enables simplifications in the ECFP best response computation. (See Section VIII for a detailed illustration.)

In contrast to FP, ECFP, in general, admits reduced complexity best response computation at the players, which is enabled by the "i.i.d." nature of the $\bar q_{-i}(t)$ used in the ECFP best response computation (see \eqref{v_m_eq}). This same factor may enable an explicit characterization of the distribution of certain essential \emph{statistics} involved in the mixed utility computation in terms of parametric statistical families. The FP best response \eqref{FP_BR}, in contrast, is based on an optimization involving the collection $q_{-i}(t)$ of individual empirical distributions; since the individual distributions are generally different, the collection $q_{-i}(t)$ does not admit a similar simplification or reduced parameterization as far as the best response computation is concerned. For more detailed illustrations of the relative complexities of FP and ECFP best response computations, we refer the reader to Section VIII-B.

In a distributed-information setting (see Section \ref{sec2A_1} and Section \ref{sec4}) players are unable to directly observe the actions of others, and therefore may not have precise knowledge of $\bar q(t)$. However, they may form an estimate of $\bar q(t)$ by exchanging information with a local subset of \emph{neighboring} players. Similarly, in general, even in more centralized information settings, due to other forms of uncertainty the $\bar q(t)$ may not be tracked exactly at each player.

 In general, we denote by $\hat q_i(t)$ the estimate which player $i$ maintains of $\bar q(t)$. Let $\varepsilon_i(t) = \| \hat q_i(t) - \bar q(t) \|$ be the error in player $i$'s estimate of $\bar q(t)$ at time $t$. In practice, since $\bar q(t)$ is not available, the player $i$ uses its estimate $\hat q_i(t)$ as a surrogate in the best response computation \eqref{v_m_eq}. However, in order to ensure convergence of the ECFP algorithm in such erroneous best response computation environments, we require that the following assumption be satisfied.
\begin{assumption}
$\varepsilon_{i}(t) = O(\frac{\log{t}}{t^r}),\mbox{ for some }r>0.$
\label{ass3}
\end{assumption}
Under this assumption, players' estimates, $\hat q_i(t)$, are asymptotically `close enough' to the true empirical centroid $\bar q(t)$ to ensure the algorithm converges. We emphasize that the exact manner in which players form estimates $\hat q^i(t)$ of $\bar q(t)$ varies from one environment to another. For the specific distributed-information setting treated in Sections V and VI, we will provide a distributed graph-based consensus-based estimation mechanism by which the players can generate their estimates $\hat q_i(t)$'s and the latter will be shown to satisfy \textbf{A.3}.

By abusing notation, we will also refer to a sequence of actions $\{ a(t) \}_{t=1}^\infty$ as an ECFP process if
\begin{align}
a_i(t+1) \in \arg\max\limits_{\alpha_i \in A_i} U\left( \alpha_i, \hat q_{-i}(t)\right),
\label{mfp_br}
\end{align}
where the initial action $a(1)$ may be chosen arbitrarily.\footnote{We note that the traditional definition of the mixed utility $U(p)$, given in \eqref{mixed_U}, is defined over the domain $\Delta^n$.  The restriction of the domain to $\Delta^n$ is not necessitated by the definition; rather, it is a consequence of the traditional approach dealing only with mixed strategies $p \in \Delta^n$.  The approximated empirical distribution $\hat q_i(t) \in \mathbb{R}^m$, however, is permitted to be outside the simplex, $\Delta_i$, and may even take negative values. In this case we retain the definition of $U(p)$, given by \eqref{mixed_U}, but extend the domain to the set of all $n$-tuples of vectors in $\mathbb{R}^m$.  This adjustment of the traditional definition expands the domain to an unbounded set, but for practical purposes, we note that assumption \textbf{A.\ref{ass3}} implies $\{\hat q_i(t)\}_{t\geq 0}$ belongs to a compact set.}
Note that, in the special case in which $\hat q_i(t) = \bar q(t),~ \forall i$, \eqref{mfp_br} reduces to \eqref{v_m_eq} and this corresponds to a centralized information setting with perfect information. Finally, we note that assumption \textbf{A.3} (and its implication on erroneous best response computation) may be applicable in other imperfect information settings beyond the specific distributed-information formulation that we develop in Section V -- VI.

In summary, in \eqref{mfp_br} each player best responds using $\hat q_i(t)$ (her personal estimate of $\bar q(t)$) as the assumed mixed strategy for the other $(n-1)$ players. In ECFP, players learn a strategy which is a consensus Nash equilibrium strategy.  The result is summarized in the following theorem.
\begin{theorem}
Let $\{a(t)\}_{t=1}^\infty$ be an ECFP process as given in \eqref{mfp_br}, such that \textbf{A.\ref{ass1}} -- \textbf{A.\ref{ass3}} hold.  Then $d(\bar q^n(t),C) \rightarrow 0$ as $t \rightarrow \infty$.
\label{thrm2}
\end{theorem}
\begin{proof}
Let $\bar a(t) = \frac{1}{n}\sum_{i=1}^n a_i(t)$, where $a_i(t) \in A_i$. Let $\bar a^n(t) \in \Delta^n$ be the $n$-tuple $(\bar a(t),\ldots,\bar a(t))$.

Note that for $t\geq 1$
\begin{equation}
\bar q^n(t+1) = \bar q^n(t) + \frac{1}{t+1} \left(\bar a^n(t+1) - \bar q^n(t) \right).
\label{thrm1_eq2}
\end{equation}
Using \eqref{thrm1_eq2} we write
\begin{align}
U(\bar q^n(t+1))= U\left(\bar q^n(t) + \frac{1}{t+1}\left(\bar a^n(t+1) - \bar q^n(t) \right)\right).
\end{align}
Applying the multilinearity of $U(\cdot)$, we obtain
\begin{align}
U(\bar q^n(t+1)) & = U(\bar q^n(t)) + \frac{1}{t+1}\sum\limits_{i=1}^n U\left(\bar a_i(t+1),\bar q_{-i}(t)\right)\\
& - \frac{1}{t+1}\sum\limits_{i=1}^n U(\bar q_i(t),\bar q_{-i}(t))  + \zeta(t+1).
\end{align}
where we have explicitly written the first order terms of the expansion and collected the remaining terms in $\zeta(t+1)$. Note that the number of second order terms in the above expansion is finite and the terms are uniformly bounded since $\max_{p\in\Delta^{n}}|U(p)| < \infty$. Hence, there exists a positive constant $M$ (independent of $t$) large enough such that $|\zeta(t+1)| \leq M(t+1)^{-2}$ for all $t$. Thus,
\begin{align}
\nonumber U(\bar q^n(t+1)) & \geq U(\bar q^n(t)) + \frac{1}{t+1}\sum\limits_{i=1}^n U\left(\bar a_i(t+1),\bar q_{-i}(t)\right)\\
& - \frac{1}{t+1}\sum\limits_{i=1}^n U(\bar q_i(t),\bar q_{-i}(t)) - \frac{M}{(t+1)^2}.
\label{thrm2_eq5}
\end{align}
The permutation invariance and multilinearity of $U(\cdot)$ permits a rearranging of terms.  The notation $U([a_j(t)]_i,\bar q_{-i}(t))$ indicates the expected utility player $i$ would receive were she to use the strategy $a_j(t)$ and all other players use the strategy $\bar q(t)$:
\vspace{-1em}
{\small
\begin{align}
  \sum\limits_{i=1}^n U\left(\bar a_i(t+1),\bar q_{-i}(t)\right)  = \sum\limits_{i=1}^n U\left( \left[\frac{1}{n} \sum\limits_{j=1}^n a_j(t+1)\right]_i,\bar q_{-i}(t)  \right)\\
  = \sum\limits_{i=1}^n \frac{1}{n} \sum\limits_{j=1}^n U\left( \left[ a_j(t+1)\right]_i,\bar q_{-i}(t)  \right)\\
  = \sum\limits_{i=1}^n \frac{1}{n} \sum\limits_{j=1}^n U\left( \left[ a_j(t+1)\right]_j,\bar q_{-j}(t)  \right)
  = \sum\limits_{j=1}^n U\left( a_j(t+1),\bar q_{-j}(t)  \right).
\end{align}
\vspace{-2mm}
}
Thus,
\begin{align}
\nonumber & U\left(\bar q^n(t+1)\right) - U\left( \bar q^n(t) \right) + \frac{M}{(t+1)^2}\\
& \geq \frac{1}{t+1}\sum\limits_{i=1}^n U\left(a_i(t+1),\bar q_{-i}(t)\right)-\frac{1}{t+1}\sum\limits_{i=1}^n U(\bar q_i(t),\bar q_{-i}(t)).
 \label{thrm1_eq1}
\end{align}
Let $v_i^m():\Delta_i\rightarrow \mathbb{R} $, $v_i^m(f) := \max_{\alpha_i \in A_i} U(\alpha_i,f_{-i}) - U(f^n)$, where $f_{-i} := (f,\ldots,f)$ is the $(n-1)$-tuple and $f^n = (f,\ldots,f)$ is an $n$-tuple. Let $L_i(t+1) = v_i^m\left(\hat q_i(t)\right) - U\left(a_i(t+1),\bar q_{-i}(t)\right)$. Substituting in $L_i(t+1)$, \eqref{thrm1_eq1} becomes
\begin{align}
\nonumber U&\left(\bar q^n(t+1)\right) - U\left( \bar q^n(t) \right) + \frac{M}{(t+1)^2} + \frac{1}{t+1}\sum\limits_{i=1}^n L_i(t+1)\\
\nonumber \geq &\frac{1}{t+1}\sum\limits_{i=1}^n \left( v_i^m\left(\hat q_i(t)\right)-U\left(\bar q_i(t),\bar q_{-i}(t)\right)\right)
= \frac{\alpha_{t+1}}{t+1},
\label{thrm2_eq4}
\end{align}
where $\alpha_{t+1} := \sum_{i=1}^n \left( v_i^m \left(\hat q_i(t)\right) - U\left(\bar q_i(t),\bar q_{-i}(t)\right) \right).$
Note that $U(\cdot)$ is multilinear and therefore locally Lipschitz continuous.  As noted earlier, assumption \textbf{A.\ref{ass3}} implies that $\{\hat q_i(t)\}_{t\geq1}$ is contained in a compact subset of $\mathbb{R}^m$. Therefore, there exists a positive constant $K$ (independent of $t$), such that $\vert U(a_i(t+1),\hat q_{-i}(t)) - U(a_i(t+1),\bar q_{-i}(t)) \vert \leq  K\| (a_i(t+1),\hat q_{-i}(t)) - (a_i(t+1),\bar q_{-i}(t))\|$, for all $t$.  By assumption \textbf{A.\ref{ass3}}, $\| \hat q_{-i}(t) - \bar q_{-i}(t)\| = O(\frac{\log{t}}{t^r})$, and hence $| U(a_i(t+1),\hat q_{-i}(t)) - U(a_i(t+1),\bar q_{-i}(t)) | = O(\frac{\log{t}}{t^r})$, which, by \eqref{mfp_br}, implies $L_i(t) = O(\frac{\log{t}}{t^r})$. In particular, $\sum_{t=2}^T \frac{L_i(t)}{t} < B$ is bounded above by some $B \in \mathbb{R}$ for all $T\geq1$.  Summing over $1\leq t \leq T$ in \eqref{thrm2_eq4},
\begin{align}
& U(\bar q^n(T+1)) - U(\bar q^n(1)) + \sum\limits_{t=1}^T \frac{M}{(t+1)^2} + \sum\limits_{t=1}^T \sum\limits_{i=1}^n\frac{L_i(t+1)}{t+1}\\
& \geq \sum\limits_{t=1}^T \frac{\alpha_{t+1}}{t+1}.
\end{align}
Note that $\sum_{t=1}^T\frac{M}{(t+1)^2}$ is summable; therefore all terms on the left hand side are bounded above for all $T\geq1$, and hence it follows that $$\sum_{t=2}^T \frac{\alpha_t}{t} < \overline{B}$$ is bounded above by some $\overline{B} \in \mathbb{R}$, for all $T\geq 2$. Let $\beta_{t+1} = \sum_{i=1}^n \left[ v_i^m(\bar q(t)) - U\left(\bar q^n(t)\right)\right]$, and note that, by definition of $v_{i}^m(\cdot)$, $\beta_{t} \geq 0$ for all $t$. By Lemma \ref{IR6} in the appendix, $| v_i^m(\hat q_i(t)) - v_i^m(\bar q(t)) | =  O \left( \frac{\log{t}}{t^r} \right)$.  Thus,
$$| \alpha_t - \beta_t | = O \left( \frac{\log{t}}{t^r} \right)$$ and hence by Lemma \ref{IR8}, $\sum_{t=2}^T \frac{\beta_t}{t} < \infty$ converges as $T\rightarrow\infty$.  By Lemma \ref{IR1} it then follows that
\begin{equation}
\lim\limits_{T \rightarrow \infty} \frac{\beta_2 + \beta_3 + \ldots + \beta_T}{T} = 0.
\end{equation}
Subsequently, by Lemma \ref{IR4}, we obtain for every $\varepsilon > 0$,
\begin{equation}
\lim\limits_{T \rightarrow \infty} \frac{ \#\{ 1 \leq t \leq T: \bar q^n(t) \notin C_\varepsilon \}}{T} = 0.
\end{equation}
By Lemma \ref{IR7}, this is equivalent to
\begin{equation}
\lim\limits_{T \rightarrow \infty} \frac{ \#\{ 1 \leq t \leq T: \bar q^n(t) \notin B_\delta(C) \}}{T} = 0
\end{equation}
for every $\delta > 0$.  Finally, by Lemma \ref{IR3}, we obtain $d(\bar q^n(t),C) \rightarrow 0$ as $t \rightarrow \infty$.
\end{proof}

We emphasize that Theorem \ref{thrm2} shows that the $n$-tuple of the {\it{average}} empirical distribution converges to $C$, that is, $d(\bar q^n(t),C) \rightarrow 0$.  This is not the same as the more traditional definition of convergence in empirical frequency,
\begin{equation}
d(q(t),C)\rightarrow 0 \mbox{ as } t \rightarrow \infty.
\label{convergence_def}
\end{equation}
In the former, the $n$-tuple containing repeated copies of the empirical centroid converges to equilibrium, in the latter, the tuple of individual empirical distributions converges to equilibrium.

The practical meaning of Theorem \ref{thrm2} is that players do in fact learn a consensus equilibrium strategy.  It is true that each player $i$ has access only to the distribution $\hat q_i(t)$.  However, the tuple of these distributions $(\hat q_1(t),\hat q_2(t),\ldots,\hat q_n(t))$ also converges asymptotically to the set of consensus equilibria (see \textbf{A.3}), i.e., $$d((\hat q_1(t),\hat q_2(t),\ldots,\hat q_n(t)),C)\rightarrow 0 \mbox{ as } t \rightarrow \infty,$$ by \textbf{A.\ref{ass3}}. Therefore, player $i$ has direct access to her portion of the convergent joint strategy. Thus each player $i$ learns a strategy which is a Nash (consensus) equilibrium with respect to the strategies learned by other players.

Note also the set of limit points of ECFP is restricted to $C$---a subset of the NE. Thus, if Pareto superior Nash equilibria exist outside the set $C$, then ECFP will never reach these points, though an algorithm such as FP may. This may be seen as a tradeoff for the improvements in memory and complexity achieved in ECFP.

\section{Distributed-Information Setting}
\label{sec4}
In a centralized-information setting (with perfect information), players are assumed to have instantaneous knowledge of the action histories of all other players. Such an assumption is clearly impractical in a large-scale scenario.  We refer to a setting where players are unable to directly observe the actions of all opponents but are equipped with an underlying communication infrastructure through which they can communicate with a local subset of neighboring players, as a distributed-information setting. This is the framework used by the distributed algorithms considered in this paper.

In order to explicitly account for the communication costs that are to be incurred in the information gathering/dissemination process, we assume a problem framework wherein players are permitted to communicate via a preassigned (sparse but connected) communication graph $G = (V,E)$.
Formally, we assume:
\begin{assumption}
Players are endowed with a preassigned communication graph $G=(V,E)$, where the vertices $V$ represent the players and the edge set $E$ consists of communication links (bidirectional) between pairs of players that can communicate directly. The graph $G$ is connected.
\label{ass11}
\end{assumption}
\begin{assumption}
Players directly observe only their own actions.
\label{ass12}
\end{assumption}
\begin{assumption}
A player may exchange information with immediate neighbors, as defined  by $G$, at most once for each iteration or round of the repeated play.
\label{ass13}
\end{assumption}
Communication is treated as non-strategic --- players do not manipulate the information they send for strategic gain. Also, we emphasize that this is a $1$-time step approach --- game play and communication take place at the same rate, i.e., evolve according to the same discrete time clock. Note that if players could
communicate arbitrarily often between game plays then, in a certain sense, we would be back to the centralized setting since infinite rounds of consensus deliver $\overline q(t)$ at each player. By restricting ourselves to one round of consensus per game play, we face a more realistic and challenging scenario.

In a distributed-information implementation of ECFP, subsequent to each round of the repeated play, players exchange information once with immediate neighbors to form an updated estimate, $\hat q_i(t)$. The next stage action $a_i(t+1)$ is then chosen as a best response to this estimate according to \eqref{mfp_br}.

The exact manner in which players update their estimate $\hat q_i(t)$ is a question of algorithmic design. The important factor is that the estimates be formed in a way such that assumption \textbf{A.\ref{ass3}} is satisfied. In section \ref{sec5} we present a distributed-information implementation of ECFP where players update $\hat q_i(t)$ according to a consensus-type algorithm.

\section{Distributed Implementation of ECFP}
\label{sec5}
\subsection{Distributed Problem Formulation}
\label{sec5-1}
We present an implementation of the ECFP algorithm which utilizes the underlying communication infrastructure presented in section \ref{sec4}; we refer to the algorithm as distributed ECFP.

Define the following two matrices:
\begin{equation}
Q(t) := ( q_1(t)~ q_2(t)~ \ldots ~ q_n(t))^T \in \mathbb{R}^{n \times m},
\end{equation}
\begin{equation}
\hat Q(t) := (\hat q_1(t) ~ \hat q_2(t) ~ \ldots ~ \hat q_n(t))^T \in \mathbb{R}^{n \times m}
\end{equation}
where $\hat q_i(t) \in \mathbb{R}^m$ denotes player $i$'s estimate of $\bar q(t) \in \mathbb{R}^m$. Let $\hat q(t) \in \mathbb{R}^n$ be the $n$-tuple $(\hat q_1(t),\ldots,\hat q_n(t))$. The tuple $\hat q(t)$ will be important in distributed ECFP; in particular we will prove that $\hat q(t)$ converges to the set of consensus equilibria.
\subsection{Distributed ECFP Algorithm}
\label{dist_mfp}
{\it{Initialize}}\newline
(i) At time $t=1$, each player $i$ chooses an arbitrary initial action $a_i(1)$.  The initial empirical distribution for player $i$ is given by $q_i(1) = a_i(1)$. Player $i$ initializes her local estimate of the empirical distribution as
\begin{equation}
\hat q_{i}(1) = \sum\limits_{j\in\Omega_{i}\cup\{i\}} w_{ij} q_j(1)
\label{ECFP_initial}
\end{equation}
where $\Omega_{i}$ is the set of neighbors of player $i$ and the $w_{ij}$'s are constant neighborhood weighting factors.

{\it{Iterate}}\newline
(ii) At each time $t>1$, player $i$ computes the set of best responses using $\hat q_i(t)$ as the assumed mixed strategy for each of the $n-1$ other players. The next action
\begin{equation}
a_i(t+1) \in \{ \arg\max_{\alpha_i\in A_i} U(\alpha_i, \hat q_{-i}(t)) \}
\label{dist_ECFP_process2}
\end{equation}
is played according to the best response calculation.   In the event of multiple pure strategy best responses, any of the maximizing actions in \eqref{dist_ECFP_process2} may be chosen arbitrarily.  The local empirical distribution $q_i(t+1)$ is updated to reflect the action taken, i.e.,

\begin{equation}
q_i(t+1) = q_i(t) + \frac{1}{t+1}(a_i(t+1) - q_i(t)).
\end{equation}

(iii) Subsequently each player $i$ computes a new estimate of the network-average empirical distribution using the following update rule:\footnote{Note that \eqref{update_eq_1} is equivalent to $\hat q_i(t+1) = \sum_j w_{i,j}(\hat q_j(t) + 1/(t+1)(a_i(t+1)-q_i(t)))$. This is closely related to minimizing an aggregate cost function using a distributed stochastic gradient descent method \cite{chen2012diffusion,chen2013distributed,ram2010distributed,srivastava2011distributed}:  $J^{glob}(q)= \sum_{i=1}^n E\|a_i(t) - q\|^2$ whose minimizer is the average of actions over time and over players, and where the expectation is over the empirical distribution of $a_i(t)$ over time. The key thing to note is that the exact minimizer of this cost is time-varying and, given that we are operating in a distributed environment with only one round of communication allowed per time slot, we can only track this dynamic minimizer using an iterative method as given in \eqref{update_eq_1}.

}
\begin{equation}
\hat q_i(t+1) = \sum\limits_{j\in\Omega_{i}\cup\{i\}} w_{i,j}\left( \hat q_j(t) + q_j(t+1) - q_j(t)\right),
\label{update_eq_1}
\end{equation}
where $\Omega_i$ is the set of neighbors of player $i$, and $w_{i,j}$ is a weighting constant.\footnote{Note that the set $\Omega_{i}\cup\{i\}$ in the summation indicates that player $i$ uses her own (local) information and that of her neighbors to update her estimate. The update rule is clearly distributed as information exchange is restricted to neighboring players only.}\newline
The update in \eqref{update_eq_1} is represented in more compact notation as
\begin{equation}
\hat Q(t+1) = W\left(\hat Q(t) + Q(t+1) - Q(t)\right),
\label{W_matrix_eq}
\end{equation}
where $W \in \mathbb{R}^{n \times n}$ is the weighting matrix with entries $w_{i,j}$.  We assume $W$ satisfies the following assumption:
\begin{assumption}
\label{ass5}
The weight matrix $W$ is an $n \times n$ matrix that is doubly stochastic, aperiodic, and irreducible, with sparsity conforming to the communication graph $G$.
\end{assumption}
Note that given assumption \textbf{A.\ref{ass11}} ($G$ is a connected graph), it is always possible to find a matrix $W$ satisfying these conditions (see \cite{dimakis2010gossip,olfati2007consensus,jadbabaie2003coordination}).

\subsection{Distributed ECFP: Main Result}
It is important to note that the process generated by \eqref{ECFP_initial} -- \eqref{update_eq_1} constitutes a special case of the general ECFP best response dynamics given in \eqref{mfp_br}, in which the estimates $\hat{q}_{i}(t)$ follow the construction \eqref{update_eq_1}. Thus, the sequence of actions $\{a(t)\}_{t=1}^{\infty}$ generated by \eqref{ECFP_initial}-\eqref{update_eq_1} is an ECFP process (in the sense of \eqref{mfp_br}), although to emphasize the distributed setting and constructions, we will refer to it as a {\it{distributed ECFP process}}.  In a distributed ECFP process, players learn a consensus equilibrium strategy in a setting where information exchange is restricted to a local neighborhood of each agent.  The result is summarized in the following Theorem.
\begin{theorem}
Let $\{a_i(t)\}_{t=1}^\infty$ be a distributed ECFP process such that assumptions \textbf{A.\ref{ass1}}--\textbf{A.\ref{ass2}}, \textbf{A.4}--\textbf{A.\ref{ass5}} hold.  Then $d(\hat q(t),C) \rightarrow 0 \mbox{ as } t \rightarrow \infty$.  In particular, the agents' estimates $\hat q_i(t)$ reach asymptotic consensus, i.e. $d\left(\hat q_i(t), \hat q_j(t)\right) \rightarrow 0 \mbox{ as } t \rightarrow \infty$ for each pair $(i,j)$ of agents.  Moreover, the agents achieve \emph {asymptotic strategy learning}, in the sense that $d((\hat q_i(t))^n,C) \rightarrow 0 \mbox{ as } t \rightarrow \infty$ for all $i = 1,\ldots,n$.
\label{thrm5}
\end{theorem}
This result implies that the $n$-tuple  $(\hat q_1(t),\ldots,\hat q_n(t))$ converges to the set $C$; since $\hat q_i(t)$ is available to player $i$, player $i$ learns the component of the consensus equilibrium strategy relevant to her.
\begin{proof}
We would like to apply the results of Theorem \ref{thrm2} to the distributed ECFP process. Clearly, by hypothesis, assumptions  \textbf{A.\ref{ass1}} and \textbf{A.\ref{ass2}} hold and, as noted above, the distributed ECFP process is in fact an ECFP process in the sense of \eqref{mfp_br}.  By Lemma \ref{lemma2}, the error in a distributed ECFP process decays as
\mathtoolsset{showonlyrefs=false}
\begin{equation}
\|\hat q_i(t) - \bar q(t)  \| = O\left( \frac{\log{t}}{t} \right),
\label{thrm_dist_ECFP_eq1}
\end{equation}
thus \textbf{A.\ref{ass3}} is satisfied (with $r=1$), and hence, the distributed ECFP fits the template of Theorem \ref{thrm2}.  Applying Theorem \ref{thrm2} yields, $d(\bar q^n(t),C) \rightarrow 0 \mbox{ as } t \rightarrow \infty$. By Lemma \ref{lemma2} we obtain, $\| \hat q_i(t) - \bar q(t) \| \rightarrow 0 \mbox{ as } t \rightarrow 0$, and the result $d(\hat q(t),C) \rightarrow 0$ as $t \rightarrow \infty$ follows.
\mathtoolsset{showonlyrefs=true}
\end{proof}

Again, we emphasize that this mode of convergence is not the same as the more traditional convergence in empirical frequency (cf. \eqref{convergence_def}, and proceeding discussion).

\section{Generalizations}
\label{sec6}
\subsection{ECFP in Permutation Invariant Potential Games}
\label{sec6a}
The assumption \textbf{A.\ref{ass2}} of identical permutation invariant utility functions can be relaxed in lieu of the following broader assumption:
\begin{assumption}
The game $\Gamma$ is an exact potential game with a permutation invariant potential function.
\label{ass6}
\end{assumption}
A game $\Gamma$ is an exact potential game if there exists some function $\Phi(y):Y^n\rightarrow \mathbb{R}$, such that
\begin{align}
& u_i(y_i',y_{-i}) - u_i(y_i'',y_{-i})\\
& = \Phi(y_i',y_{-i}) - \Phi(y_i'',y_{-i})\; \forall i \in N,  \forall y_i',y_i'' \in Y_i.
\end{align}
The function $\Phi(y)$ is called a potential function for $\Gamma$.  The generalized form of Theorem \ref{thrm2} is as follows:
\begin{theorem}
Let $\{a(t)\}_{t=1}^\infty$ be an ECFP process such that \textbf{A.\ref{ass1}} (identical action spaces), \textbf{A.\ref{ass3}} ($\varepsilon_i(t) = O\left( \frac{\log{t}}{t^{r}} \right)\; ~\forall i$, for some $r>0$), and \textbf{A.\ref{ass6}} hold.  Then $d(\bar q^n(t),C) \rightarrow 0$ as $t \rightarrow \infty$.
\label{thrm3}
\end{theorem}
\begin{IEEEproof}
Let $\Gamma_1 = \left( N,Y,\{U_i\}_{i \in N}\right)$ be an exact potential game with potential function $\Phi$.  Let $\Gamma_2 = \left( N,Y,\{\tilde U_i\}_{i \in N}\right)$ be a game with the same set of players and actions as $\Gamma_1$, but with all players using $\Phi$ as their utility function ($\tilde U_i = \Phi,~\forall i$). Let $C_{\Gamma_1}$ and $C_{\Gamma_2}$ be the set of consensus equilibria in $\Gamma_1$ and $\Gamma_2$ respectively.  Let $\bar q_{\Gamma_1}(t),~\bar q_{\Gamma_2}(t)$ be the average empirical distributions corresponding to ECFP processes in $\Gamma_1$ and $\Gamma_2$ respectively. Note that the set of consensus equilibria in $\Gamma_1$ and $\Gamma_2$ coincide \cite{Mond96}.  Also note that $\Gamma_1$ and $\Gamma_2$ are best response equivalent \cite{Voor00}, therefore a valid ECFP process for $\Gamma_1$ is a valid ECFP process for $\Gamma_2$, and vice versa. $\Gamma_2$ is a game with identical action spaces and identical permutation-invariant utility functions and therefore falls within the purview of Theorem \ref{thrm2}. By Theorem \ref{thrm2}, $d(\bar q_{\Gamma_2}^n(t),C_{\Gamma_2}) \rightarrow 0$. By best response equivalence, any valid ECFP process in $\Gamma_1$ is a valid ECFP process in $\Gamma_2$, therefore $d(\bar q_{\Gamma_1}^n(t),C_{\Gamma_2}) \rightarrow 0$. Since $C_{\Gamma_1}$ and $C_{\Gamma_2}$ coincide, $d(\bar q_{\Gamma_1}^n(t),C_1) \rightarrow 0$.
\end{IEEEproof}
Potential games are studied in \cite{Mond96}. A game that admits an exact potential function is known as an exact potential game.  The class of exact potential games includes congestion games \cite{Rose73}.  Congestion games have many useful applications in economics and engineering. We present an example of a congestion game in a distributed cognitive radio context in Section VIII.

\subsection{ECFP in Potential Games with Permutation Invariant Classes}
The ECFP algorithm may be generalized to a setup where players track multiple centroids, each centroid corresponding to a different class of players. This generalization allows for ECFP to be considered in more general classes of games and allows for FP to be considered as special case of ECFP.
Formally, assume
\begin{assumption}
$\Gamma$ is an exact potential game with potential function $\Phi$.
\label{ass12}
\end{assumption}
Let the players be partitioned into classes as follows. For $m\leq n$, let $I = \{1,\ldots,m\}$, and let $\mathcal{P}=\{P_1,\ldots,P_m\}$ be a collection of subsets of $N$; i.e. $P_k \subseteq N, ~\forall k\in I$. A collection $\mathcal{P}$ is said to be a \textit{permutation invariant partition} of $N$ if,\\
$(i)$ $P_k \cap P_\ell = \emptyset$, for $k,\ell \in I$, $k\not= \ell$, \\
$(ii)$ $\bigcup\limits_{k\in I} P_k = N$,\\
$(iii)$ for $k\in I$, $i,j\in P_k$, $Y_i = Y_j$,\\
$(iv)$ for $k\in I$, $i,j\in P_k$, there holds for any strategy profile $y' \in Y_i, ~y'' \in Y_j, ~ y_{-(i,j)} \in Y_{-(i,j)}$,
$$ \Phi([y']_i,[y'']_j,y_{-(i,j)}) = \Phi([y'']_i,[y']_j,y_{-(i,j)}).$$
\begin{remark}
A partition which places a single player in each class is always a valid permutation invariant partition. In this case, the resultant ECFP process will be equivalent to FP. Moreover, if $\Phi$ is permutation invariant, then all players may be partitioned into a single class; in this case, the resultant ECFP process will be equivalent to that  presented in Section \ref{sec6a}.
\end{remark}

For a collection $\mathcal{P}$, define $\psi(i):N\rightarrow I$ to be the unique mapping such that $\psi(i) = k$ if and only if $i\in P_k$.

Given a permutation invariant partition $\mathcal{P}$, let the set of symmetric Nash equilibra (SNE) relative to $\mathcal{P}$ be given by,\footnote{A symmetric equilibrium is typically defined as a Nash equilibrium that is invariant under any automorphism of the game. Note that, in contrast to this, the definition of SNE given here has the additional constraint that it is defined relative to the particular partition $\mathcal{P}$. If $\mathcal{P}$ is such that no player is permutation equivalent to a player in another class, then the two concepts coincide. Furthermore, if $\mathcal{P}$ is such that there is one player in each class, then the set of SNE relative to $\mathcal{P}$ will coincide with the set of NE.}
\begin{equation}
SNE := \{p\in NE: \forall~ k\in I, p_i = p_j \forall i,j \in P_k\}.
\end{equation}

For $k\in I$, define\footnote{The terms $\bar g(t)$, $\hat g_i(t)$, etc. as defined in this section are analogous to the terms defined in Table \ref{table_ECFP}. The notation is changed from $q$'s to $g$'s in this section to emphasize the differences in definition for the generalized setup.}
$$\bar g^k(t) := |P_k|^{-1}\sum_{i\in P_k} q_i(t)$$
to be the \textit{$k$th empirical centroid distribution} relative to $\mathcal{P}$. Likewise, define $\bar g(t) = (\bar g_1(t),\ldots,\bar g_n(t))$
where $\bar g_i(t) = \bar g^{\psi(i)}(t)$, to be the \textit{composite empirical centroid distribution} relative to $\mathcal{P}$.

As before, we consider a scenario where players do not have precise knowledge of the centroid distribution. Let $\hat g_i^k(t) $ be the estimate which player $i$ maintains of the $k$'th centroid $\bar g^k(t)$.  Formally, assume
\begin{assumption}
$\|\hat g_i^k(t) - \bar g^k(t)\| = O\left(\frac{\log(t)}{t}\right),~ \forall i\in \mathbb{N},~\forall k\in I$.
\label{ass13}
\end{assumption}

In this context, we say a sequence of actions $\{a(t)\}_{t\geq 1}$ is an ECFP process (with respect to $\mathcal{P})$ if
$$ U_i(a_i(t), \hat g_{-i}(t)) = \max_{\alpha_i\in A_i} U_i(\alpha_i ,\hat g_{-i}(t)),$$
where,
$$\hat g_{-i}(t)  = (\hat g_i^{\psi(1)}(t),\ldots,\hat g_i^{\psi(i-1)}(t),g_i^{\psi(i+1)}(t),\ldots\hat g_i^{\psi(n)}(t)).$$
The following theorem asserts that an ECFP process will converge, in this generalized setup, to the set of $SNE$.
\begin{theorem}
Let $\Gamma$ be a potential game, let $\mathcal{P}$ be a permutation invariant partition of the player set $N$, let  $\{a(t)\}_{t\geq 1}$ be an ECFP process with respect to $\mathcal{P}$, and assume \textbf{A.\ref{ass12}} -- \textbf{A.\ref{ass13}} hold. Then $d(\bar g(t),SNE)\rightarrow 0$ as $t\rightarrow\infty$.
\end{theorem}
The proof of this result follows the same reasoning as the proof of Theorem \ref{thrm2} and is omitted here for brevity.

A distributed-information implementation of ECFP in this generalized setup may be achieved in a manner analogous to that of Section \ref{sec4}, with the primary difference that, in this context, players exchange estimates for each empirical centroid distribution $\bar g^{k}(t)$, $k\in I$.

\section{Experimental Results}
\label{sec7}
In this section we illustrate the operation of distributed-ECFP by implementing it in a cognitive radio application.
\subsection{Cognitive Radio Setup}
Let $Ch$ indicate a finite collection of permissible frequency channels. Assume there are two classes of users sharing the allocated set of channels: primary users and secondary users. Assume each primary user has been assigned to a fixed channel from which they may not deviate. Secondary users are free to use any channel they wish. The objective in this setup is for the secondary users to cooperatively learn a channel allocation that is both fair and in some sense optimal.

Cast this setup in the format of a normal form game $\Gamma = (N,(Y_i)_{i\in \mathbb{N}}, (u_i(\cdot))_{i\in\mathbb{N}})$ with $N$ being the set of secondary users, and $Y_i = Ch$ for all $i$. Let $\sigma_r(y)$ (respectively, $\sigma_r(y_{-i})$) denote the set of users on channel $r\in Ch$ for the joint strategy $y\in Y^n$ ($y_{-i} \in Y_{-i}$).

The cost associated with channel $r$ when $k$ users are on channel $r$ is given by $c_r(k)$. The utility for player $i$ is given by $u_i(y) = -c_{y_i}(\sigma_{y_i}(y))$,
and the mixed utility is given by the usual multilinear extension. The game $\Gamma$ is an instance of a congestion game---a known subset of potential games---and hence is amenable to ECFP  by Theorem \ref{thrm3}.

\subsubsection{Communication Graph Setup}
We assume that some small portion of spectrum is allocated for the purpose of transmitting data pertinent to the learning algorithm (i.e., disseminating information about the empirical centroid $\bar q(t)$.)
Such an assumption is reasonable when the communication overhead associated with the learning algorithm is relatively small compared to the objective data being transmitted, e.g., users objective is to transmit large video files.

We model user-to-user communication using a geometric random graph. In implementing the ECFP algorithm of Section \ref{sec5}, we assign the weight constants $w_{i,j}$ of \eqref{update_eq_1} according to the Metropolis-Hastings rule \cite{chib1995understanding}.

\subsection{Best Response Computation}
\label{sec_BR_computation}
On the surface, the ECFP best response calculation \eqref{v_m_eq} appears to have the same complexity as the FP best response calculation \eqref{FP_BR}.  However, the symmetry inherent in the distribution $\bar q_{-i}(t)$ used in the ECFP best response calculation leads to mitigations in computational complexity. We contrast the FP and ECFP best response computations for the case of the cognitive radio game.
\subsubsection{ECFP Best Response Computation}
In order to choose a best response in ECFP, a player must compute\footnote{In the ECFP best response \eqref{v_m_eq}, a player must maximize the mixed utility $\max_{\alpha_i \in A_i} U(\alpha_i,\bar q_{-i}(t))$. Recall that the mixed utility \eqref{mixed_U} is the expected value of $u(\cdot)$ given that players are using probabilistic strategies $\bar q_{-i}(t)$. Thus, maximizing the mixed utility of \eqref{v_m_eq} is equivalent to maximizing the expected value below.}
$$\arg\max_{y_i \in Y_i} E_{\bar q_{-i}(t)}[u(y_i,y_{-i})],$$
where $y_{-i}$ is a random variable with distribution $\bar q_{-i}(t)$.
The symmetry in the game allows for the following simplification,
\begin{equation}
E_{\bar q_{-i}(t)}[u(y_i,y_{-i})] = \sum_{k=0}^{n-1} c_{y_i}(k+1) \mathbb{P}_{\bar q_{-i}(t)}(\sigma_{y_i}(y_{-i}) = k).
\label{ECFP_congestion_BR}
\end{equation}
In the above, players only need to compute the probability associated with having $k$ users on each channel rather than computing the probability of every possible configuration of users.

From here, the symmetry in the i.i.d. distribution $\bar q_{-i}(t)$ allows for further simplifications. Let $y_i = r\in Ch$ and note that\footnote{Recall that the notation $\bar q(t,r)$ refers to the $r$th element of the vector $\bar q(t)$. In ECFP, players (incorrectly) assume that all opponents are independently using the identical mixed strategy $\bar q(t)$. Under this assumption, the probability of any given opponent using channel $r$ is given by $\bar q(t,r)$.}
{\small
\begin{align}
\mathbb{P}_{\bar q_{-i}(t)}(\sigma_r(y_{-i}) = 0) & = (1-\bar q(t,r))^{n-1}\\
\mathbb{P}_{\bar q_{-i}(t)}(\sigma_r(y_{-i}) = 1) & = (n-1)\bar q(t,r)(1-\bar q(t,r))^{n-2}\\
\mathbb{P}_{\bar q_{-i}(t)}(\sigma_r(y_{-i}) = 2) & = \binom{n-1}{2}\bar q(t,r)^2(1-\bar q(t,r))^{n-3}.
\end{align}
}
As the above pattern suggests, the probability is binomial---for $0\leq k \leq (n-1)$,
{\small
$$\mathbb{P}_{\bar q_{-i}(t)}(\sigma_{y_i}(y_{-i}) = k) = \binom{n-1}{k}\bar q(t,r)^k(1-\bar q(t,r))^{n-1-k}.$$
}
Thus, the requisite probability is given by a computationally simple closed form expression, and the expected utility can be easily computed using \eqref{ECFP_congestion_BR}. Furthermore,
since players compute a best response each iteration, they reap the computational benefits of using this simplified form for $\mathbb{P}_{q_{-i}(t)}(\sigma_{y_i}(y_{-i}) = k)$ on each iteration of ECFP.

\subsubsection{FP Best Response Computation}
In order to choose a best response in FP, a player must compute
$$\arg\max_{y_i \in Y_i} E_{q_{-i}(t)}[u(y_i,y_{-i})],$$
where $y_{-i}$ is a random variable with distribution $q_{-i}(t)$.
As before, the symmetry in the game allows for a simplification to
$$E_{q_{-i}(t)}[u(y_i,y_{-i})] = \sum_{k=0}^{n-1} c_{y_i}(k+1) \mathbb{P}_{q_{-i}(t)}(\sigma_{y_i}(y_{-i}) = k).$$
However, due to the lack of structure in $q_{-i}(t)$, no further simplifications are possible. To illustrate the complexity of this, let $y_i = r\in Ch$ and note that
{\small
\begin{align}
\mathbb{P}_{q_{-i}(t)}(\sigma_r(y_{-i}) = 0) & = \prod_{j\not = i}(1-q_{j}(t,r))^{n-1}\\
\mathbb{P}_{q_{-i}(t)}(\sigma_r(y_{-i}) = 1) & = \sum_{j_1\not = i} q_{j_1}(t,r)\prod_{j_2 \not = i,j_1}(1-q_{j_2}(t,r))^{n-2}\\
\mathbb{P}_{q_{-i}(t)}(\sigma_r(y_{-i}) = 2) & =\\
  \sum_{j_1 \not= i}\sum_{j_2 > j_1} q_{j_1}(t,r)&q_{j_2}(t,r)\prod_{j_3\not = i,j_1,j_2}(1-q_{j_3}(t,r))^{n-3}.
\end{align}}
In general, when the $q(t)$ corresponds to a mixed strategy, the complexity of evaluating $\mathbb{P}_{q_{-i}(t)}(\sigma_r(y_{-i}) = k)$ grows combinatorially  with $k$---even in this game with symmetric payoffs.

\subsection{Simulation Results}
We simulated ECFP in two different cognitive radio scenarios. In the first, there are $10$ channels and $400$ users, and the cost function for channel $r\in Ch$ is given by a cubic polynomial of the form
$$c_r(k) = a_3k^3 + a_2k^2 + a_1k + a_0,$$
where $k$ is the number of users on channel $r$ and $a_j$, $0\leq j\leq 3$ are arbitrary coefficients. Figure \ref{fig_util} show a plot of utilities $U(\bar q^n(t))$ and $U((\hat q_1(t))^n)$ in the centralized and distributed cases respectively.\footnote{The notation $(\hat q_i(t))^n$ signifies the $n$-tuple containing $n$ repeated copies of $\hat q_i(t)$. } The choice of the distribution of player $1$, $\hat q_1(t)$, to represent the distributed case was arbitrary---$\hat q_i(t)$ for any $i\in N$ produces a similar results. In the distributed case, players communicated via a randomly generated geometric graph with average node degree of $8.78$.

Both the centralized and distributed algorithms were started with identical initial conditions. It is interesting to note that, while multiple NE do exist, both algorithms tend to converge to the same equilibrium, regardless of the communication graph topology. This trend suggests that the basin of attraction for any given NE is similar for both centralized ECFP and distributed ECFP. Neither algorithm is noticeably superior in terms of the quality of equilibria attained.

A useful feature of consensus NE (CNE) in this setup is their adaptability to players entering or exiting the game. If a new player enters the game after an ECFP learning process has been running for some time, then incumbent players can simply inform the new player of the current empirical distribution $\bar q(t)$, and the distribution $\bar q^{n+1}(t)$ (meaning the $(n+1)$ tuple which contains repeated copies of $\bar q(t)$) will be an approximate CNE in the newly formed $(n+1)$ player game. Similarly, if a player exits the game, the distribution $\bar q^{n-1}(t)$ will be an approximate CNE in the newly formed $(n-1)$ player game.

In the second cognitive radio scenario simulated, there are $10$ channels, each with a quadratic cost function. This choice of cost functions guarantees the existence of a unique CNE. We simulated distributed-ECFP in this scenario for the cases of $50$, $200$, and $500$ users; each case had different randomly generated cost functions. In each case, the communication graph was generated as a random geometric graph. The average node degree for the associated communication graph in each case was $8.04$, $8.72$, and $8.98$ respectively.

Figure \ref{fig_dist} shows a plot of the normalized distance of $(\hat q_1(t))^n$ to the unique $NE$ in each case (the particular choice of $\hat q_1(t)$ again being arbitrary). Distance was measured using the Euclidean norm, normalized by $\sqrt{n}$, where $n$ is the number of players. Simulation results suggest that the convergence rate of ECFP is independent of the number of players. Indeed, the analytical properties of ECFP (in general, see Section VII) suggest that the convergence is dependent only on the number of permutation invariant classes into which the player set is partitioned and not the overall number of players. A rigorous characterization of the precise nature of this relationship may be an interesting topic for future research.


\begin{figure}[t]
\centering
 \vspace{-0.2cm}
    \subfigure[]{\includegraphics[width=4.3cm]{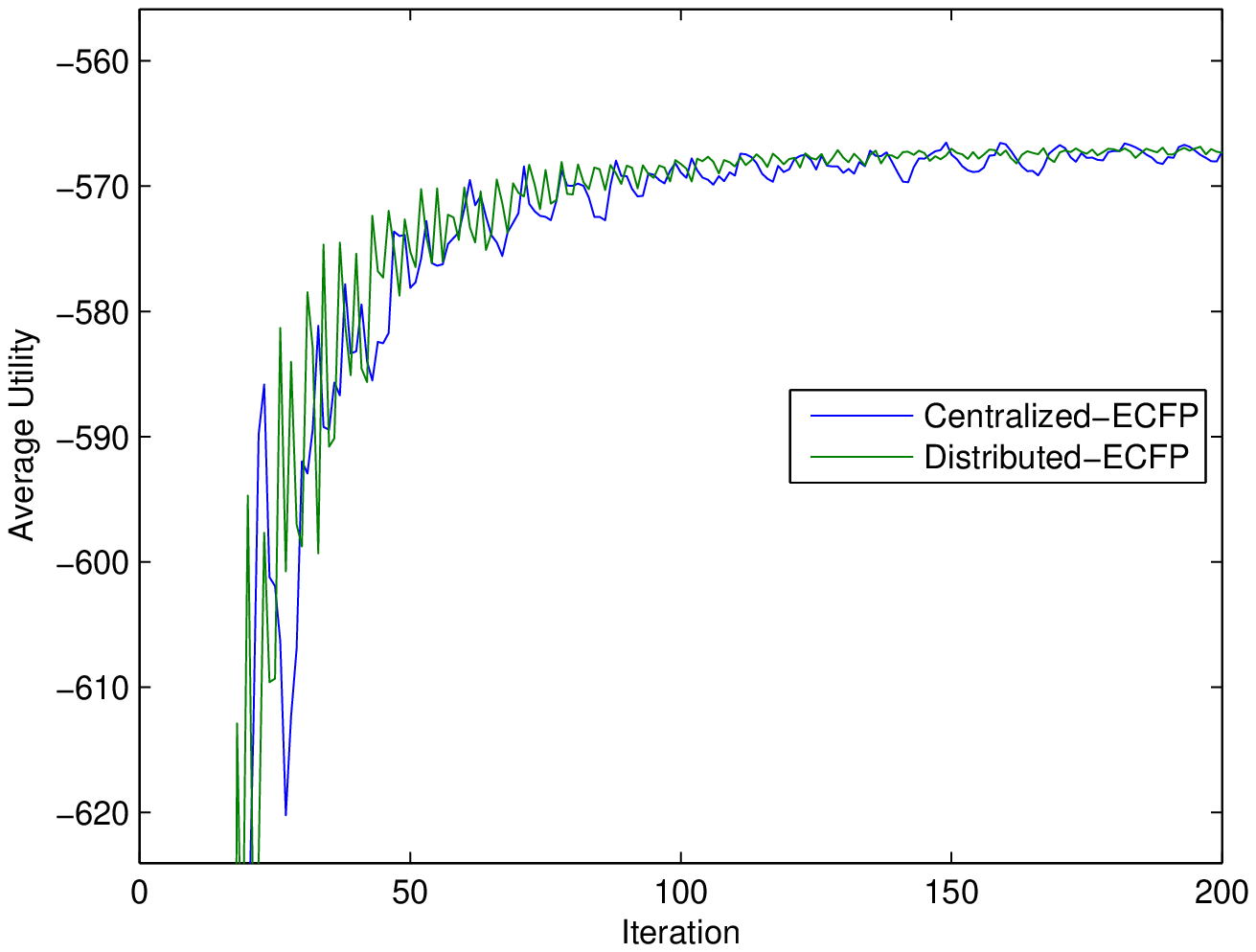}\label{fig_util}}\hspace{0.1cm}
    \subfigure[]{\includegraphics[width=4.3cm]{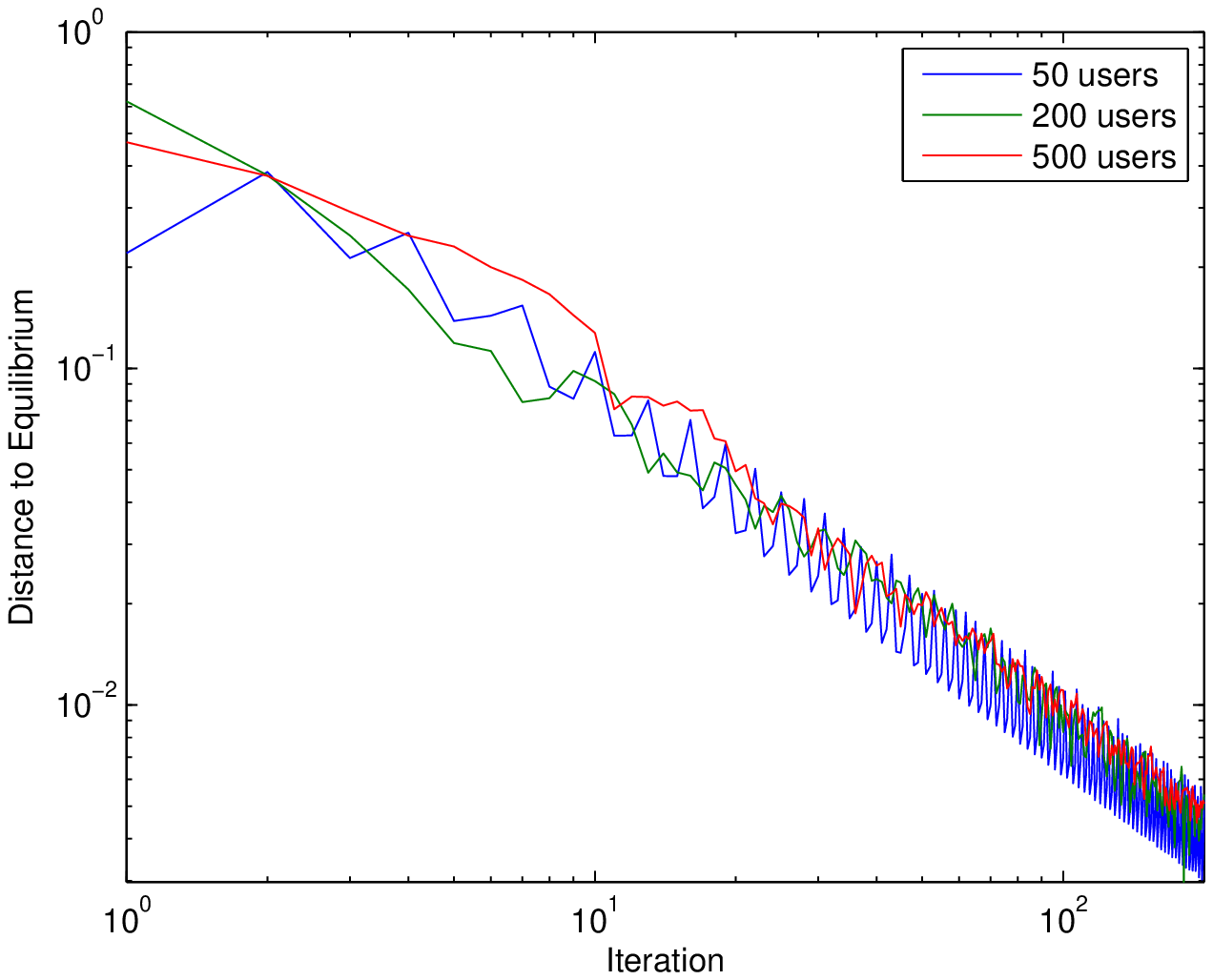}\label{fig_dist}}
\caption{(a) The average utility (taken over the set of players) of the joint empirical distribution, $\bar q^n(t)$; and (b) The distance of the joint empirical distribution to the set of NE.}\label{Figures2}
\end{figure}

\section{Conclusions}
\label{sec8}
We have introduced a variant of the well-known FP algorithm that we call empirical centroid fictitious play (ECFP).  Rather than track and respond to the empirical distribution of each opposing player, as in FP, ECFP requires that players track only the centroid of the marginal empirical distributions and compute a best response with respect to this same quantity. The memory problem associated with FP in large-scale games is mitigated by requiring players to track a distribution which is invariant to the number of players in the game. The problem of computational complexity is mitigated by the introduction of symmetry into the best response calculation. ECFP is shown to converge to the a subset of the Nash equilibria (the consensus equilibria), for potential games with permutation invariant potential functions.

In addition, we have presented a general formulation of ECFP where the player set is partitioned into classes and players track one centroid for each class.

We have introduced a distributed-information learning framework wherein it is assumed that players are unable to directly observe the actions of others but may communicate with a local subset of neighboring players via an underlying communication infrastructure. We presented an implementation of ECFP in this framework which mitigates all three problems (communication, memory, and computational complexity) associated with FP in large-scale games.

An interesting future research direction will be to investigate the convergence rate of ECFP in terms of both the number of players and the number of classes into which the player set is partitioned. It would also be of interest to investigate a distributed-information implementation of the ECFP algorithm within other communication infrastructures (e.g., random link failures, asynchronous communications).

\section*{Appendix}
{\small
\subsection{Distributed averaging in dynamic networks}\label{dist_avg_sec}

This appendix concerns topics in distributed consensus in networks where node values are dynamic quantities.  The results of this section are used to prove convergence of the distributed algorithms presented in section \ref{dist_mfp}.  Results in this section are similar to results on distributed averaging in networks with additive changes in node values and information dynamics in \cite{Raj01, kar2011convergence, chen2012diffusion}.  For a survey of traditional consensus and gossip algorithms, the reader may refer to  \cite{olfati2007consensus,shah_book,dimakis2010gossip}.

Consider a network of $n$ nodes connected through a communication graph $G=(V,E)$.  The graph is assumed to be connected.  Let $x_i(t)\in\mathbb{R}$ be the value of node $i$ at time $t$, and let $x(t)\in\mathbb{R}^n$ be the vector of values at all nodes.  The goal is for each node to track the instantaneous average $\bar x(t) = \frac{1}{n}\sum_{i=1}^n x_i(t)$, $\bar x(t) \in \mathbb{R}$, given that the value at each node $x_i(t)$ is time varying.   Let $\delta_i(t) = x_i(t+1) - x_i(t)$ be the change in the value at node $i$, and $\delta(t) = x(t+1) - x(t)$ be the vector of changes at all nodes, $\delta(t) \in \mathbb{R}^n$.  Suppose the magnitude of the change at time $t$ is bounded by $\vert \delta_i(t) \vert = \vert x_i(t+1) - x_i(t) \vert \leq \epsilon(t) ~\forall i$.  We make the following assumption:
\begin{assumption}
The sequence $\{\varepsilon(t)\}_{t=0}^\infty$ is monotone non-increasing.
\label{ass9}
\end{assumption}

Let $\hat x_i(t)\in\mathbb{R}$ be the estimate of $\bar x(t)$ at node $i$ and let $\hat x(t)\in\mathbb{R}^n$ be the vector of estimates.  We make the following assumption pertaining to the initial error in players' estimates.
\begin{assumption}
$\hat x_i(0) - \bar x(0) = 0 ~\forall i$.
\label{ass10}
\end{assumption}
Let the average be estimated using the update rule
\begin{align}
\hat x(t+1) = W \left( \hat x(t) + x(t+1) - x(t) \right),
\label{update_rule1}
\end{align}
where the matrix $W\in\mathbb{R}^{n\times n}$ is aperiodic, irreducible, and doubly stochastic with sparsity conforming to $G$.  The following Lemma gives a bound for the error in the estimates of $\bar x(t)$.
\begin{lemma}
Let the sequence $\{\hat x(t)\}_{t=1}^\infty$ be computed according to \eqref{update_rule1} such that assumptions \textbf{A.\ref{ass11}}, \textbf{A.\ref{ass5}}, and \textbf{A.\ref{ass10}} hold and let the incremental change in $x(t)$ be bounded according to assumption \textbf{A.\ref{ass9}}.  Then the error at any time $t$ is bounded by,
$$\| \hat x(t) - \bar x(t) \ones \| \leq \frac{2\sqrt{n}}{1-\lambda}\epsilon_{avg}(t), $$
where $\lambda = \sup\limits_{y \in \mathbb{R}^n : \sum\limits_i y_i = 0} \frac{\| Wy \|}{\|y\|}$, and $\epsilon_{avg}(t) = \frac{1}{t}\sum\limits_{\tau=0}^{t-1} \epsilon(\tau)$ is the time average of $\{ \epsilon(\tau) \}_{\tau=0}^{t-1}$.
\label{lemma4}
\end{lemma}
\begin{proof}
Let $e(t) = \hat x(t) - \bar x(t)\ones$ be the vector of errors in each players estimate of $\bar x(t)$, where $\mathbf{1}$ denotes the $n\times 1$ vector of all ones. Let $\bar \delta (t) = \frac{1}{n}\sum_{i} \delta_i(t),~\forall t.$
Using the relation \eqref{update_rule1} and the properties of doubly stochastic matrices, the vector of errors may be written recursively as,\vspace{-3mm}
\begin{align}
e(t+1) & = W \left( e(t) + \xi(t) \right)
\label{thrm4_eq1}
\end{align}
where $\xi(t) = \delta(t) - \bar\delta(t) \ones $.  Note that
\begin{align}
 \vert \xi_i(t) \vert & = \vert \delta_i(t) - \bar\delta_i(t) \vert \leq \vert \delta_i(t) \vert + \vert \bar\delta_i(t) \vert \leq 2\epsilon(t),
 \end{align}
and,
\begin{equation}
\|\xi(t)\|^2 = \sum\limits_{i=1}^n (\xi_i(t))^2 \leq \sum\limits_{i=1}^n 4\epsilon(t)^2 = 4n\epsilon(t)^2.
\label{xi_norm}
\end{equation}

Using \eqref{thrm4_eq1}, the error $e(t)$ can be rewritten as a function of $\xi(t)$ and $e(0)$, that is $e(t+1) = \sum\limits_{r=0}^{t} W^{r+1}\xi(t-r) + W^{t+1}e(0).$
Using this relationship we establish an upper bound on the error,
\begin{align}
\nonumber\|e(t+1) \| = \| \sum\limits_{r=0}^t W^{r+1}\xi(t-r) + W^{t+1}e(0)\| \\ \leq \sum\limits_{r=0}^t \lambda^{r+1} \| \xi(t-r) \|,
\label{thrm5_eq2}
\end{align}
where we have employed assumption \textbf{A.\ref{ass10}}, $e(0) = 0$. Applying \eqref{xi_norm} in \eqref{thrm5_eq2}, we get $\|e(t+1)\| \leq \sum\limits_{r=0}^t \lambda^{r+1} 2\sqrt{n}\epsilon(t-r).$

Recall that $\varepsilon_{avg}(t) = \frac{1}{t}\sum\limits_{\tau=0}^{t-1} \epsilon(\tau)$ is the time average of the sequence $\{\varepsilon(t)\}$ up to time $t$, and note that given our assumptions on $W$, it holds that $\lambda <1$ (see \cite{dimakis2010gossip}).  Note that, by Chebychev's sum inequality \cite{hardy1988inequalities} (p. 43-44),
\begin{align}
\sum\limits_{r=0}^t \lambda^{r+1} 2\sqrt{n}\epsilon(t-r) \leq \sum\limits_{r=0}^t \lambda^{r+1} 2\sqrt{n} \varepsilon_{avg}(t+1),
\end{align}
and hence,
\begin{align}
\|e(t+1)\| & \leq \sum\limits_{r=0}^t \lambda^{r+1} 2\sqrt{n} \varepsilon_{avg}(t+1)\\
& = \left( \lambda\frac{1 - \lambda^{t+1}}{1-\lambda} \right)2\sqrt{n}\epsilon_{avg}(t+1)  \leq \frac{2\sqrt{n}}{1-\lambda}\epsilon_{avg}(t+1),
\end{align}
giving the desired upper bound for the error.
\end{proof}

\begin{lemma}
Let $\{ a(t) \}_{t\geq1}$ be a distributed ECFP process as defined in section \ref{dist_mfp} (see equations \eqref{W_matrix_eq}-\eqref{dist_ECFP_process2}). Then $\|\hat q_i(t) - \bar q(t) \| = O(\frac{\log{t}}{t})$, where $\bar q(t)$ is the average empirical distribution and $\hat q_i(t)$ is player $i$'s estimate of $\bar q(t)$.
\label{lemma2}
\end{lemma}
\begin{proof}
We use the second argument, $k$, to index the components of the vector $q_i(t) \in \mathbb{R}^m$. Noting that
\begin{equation}
q_i(t+1) = q_i(t) + \frac{1}{t+1} \left( a_i(t+1) -  q_i(t) \right),
\end{equation}
it follows that the maximum incremental change for any single value in the vector $q_i(t)$ is bounded by $\vert q_i(t+1,k) - q_i(t,k)\vert \leq \frac{1}{t+1} = \varepsilon(t)$, where we let $\epsilon(t) := \frac{1}{t+1}$. Note that the distributed ECFP process \eqref{W_matrix_eq} is updated column-wise (each column corresponds to an action $k$) using an update rule equivalent to \eqref{update_rule1} of Lemma \ref{lemma4}.  Also note that, column-wise, all necessary conditions of Lemma \ref{lemma4} are satisfied,\footnote{The assumption of zero initial error (\textbf{A.\ref{ass10}}) is satisfied since the initialization of $\hat q_i(1)$ in \eqref{ECFP_initial} is equivalent to letting $\bar q_i(0) = 0$, $\hat q_i(0) = 0$ for all $i$ in \eqref{update_eq_1}.} and specifically, we have $\epsilon(t) = \frac{1}{t+1}$. Thus we apply Lemma \ref{lemma4} column-wise to $\hat Q$ and $Q(t)$ of \eqref{W_matrix_eq}, where $x(t)$ of Lemma \ref{lemma4} corresponds to the $k$'th column of $Q(t)$, and $\hat X(t)$ of Lemma \ref{lemma4} corresponds to the $k$'th column of $\hat Q(t)$, and obtain
\begin{equation}
\|\hat{q}(t,k)- \bar q(t,k) \ones \| \leq \frac{2 \sqrt{n}}{1-\lambda}\epsilon_{avg}(t) =  O\left( \frac{\log{t}}{t}\right),
\label{thrm6_eq1}
\end{equation}
where $\hat{q}(t,k)=(\hat{q}_{1}(t,k)~\hat{q}_{2}(t,k)\cdots\hat{q}_{n}(t,k))^{T}$ and $\epsilon_{avg}(t) = \frac{1}{t}\sum_{s=1}^t \frac{1}{t+1} = O\left( \frac{\log{t}}{t}\right)$.  Thus, $ \vert \hat q_i(t,k) - \bar q(t,k) \vert = O\left( \frac{\log{t}}{t} \right),\; k = 1,\ldots,m,\; \forall i$, and hence, $\| \hat q_i(t) - \bar q(t) \| = O(\frac{\log{t}}{t}),~\forall i.$
\end{proof}

\subsection{Intermediate Results}
\begin{lemma}
Suppose the sum $-\infty < \sum\limits_{t=1}^\infty \frac{a_t}{t} = S < \infty$ converges, then $\lim\limits_{T \rightarrow \infty} \frac{a_1 + a_2 + \ldots + a_T}{T} = 0$.
\label{IR1}
\end{lemma}
\begin{IEEEproof}
By Kronecker's Lemma \cite{shiryaev1996probability}, $-\infty < \sum\limits_{k=1}^\infty \frac{a_k}{k} = S < \infty \Rightarrow \lim\limits_{T \rightarrow \infty} \frac{1}{T}\sum\limits_{k=1}^T k\frac{a_k}{k} = 0$, which implies that $\lim\limits_{T \rightarrow \infty} \frac{a_1 + \ldots + a_t}{T} = 0.$
\end{IEEEproof}
\begin{lemma}
For $i\in N$, let $\{q_{-i}(t)\}_{t=1}^\infty\in\Delta_{-i}$ and $\{r_{-i}(t)\}_{t=1}^\infty\in\Delta_{-i}$ be sequences such that $\|q_{-i}(t) - r_{-i}(t)\| = O\left(\frac{\log{t}}{t^r} \right),~r>0$.
Let $U(p):\Delta^n\rightarrow\mathbb{R}$ be the (multilinear) mixed utility function defined in \eqref{mixed_U}. Then
\begin{equation}
\vert \max\limits_{p_i \in \Delta_i} U(p_i,q_{-i}(t)) - \max\limits_{p_i \in \Delta_i} U(p_i,r_{-i}(t)) \vert = O\left( \frac{\log{t}}{t^r} \right).
\end{equation}
\label{IR6}
\end{lemma}
\begin{proof}
Let $\zeta_{-i}' \in \Delta_{-i}$ and $\zeta_{-i}'' \in \Delta_{-i}$.  Let $p^* = \arg\max_{p_i \in \Delta_i} U(p_i,\zeta_{-i}')$ and $p^{**} = \arg\max_{p_i \in \Delta_i} U(p_i,\zeta_{-i}'')$.\footnote{Note that such $p^{\ast}$ and $p^{\ast\ast}$ exist, as $U(\cdot)$ is continuous and the maximization set is compact.}  $U(\cdot)$ is multilinear and is therefore Lipschitz continuous over the domain $\Delta^n$.  Let $K$ be the Lipschitz constant for $U(\cdot)$ such that $\vert U(x) - U(y) \vert \leq K\|x-y\|$ for $x,y \in \Delta^n$. By Lipschitz continuity, it holds that
\begin{align}
\nonumber U(p^*,\zeta_{-i}') & \leq U(p^*,\zeta_{-i}'') + K\|\zeta_{-i}' - \zeta_{-i}'' \| \\
& \leq U(p^{**},\zeta_{-i}'') + K\|\zeta_{-i}' - \zeta_{-i}'' \|,
\label{IR6_redux_eq1}
\end{align}
and thus $U(p^{*},\zeta_{-i}') - U(p^{**},\zeta_{-i}'') \leq K\|\zeta_{-i}' - \zeta_{-i}'' \|$.  By a symmetric argument to \eqref{IR6_redux_eq1}, we also establish $U(p^{**},\zeta_{-i}'') - U(p^{*},\zeta_{-i}') \leq K\|\zeta_{-i}' - \zeta_{-i}'' \|$, thus $\vert U(p^{*},\zeta_{-i}') - U(p^{**},\zeta_{-i}'') \vert \leq K\|\zeta_{-i}' - \zeta_{-i}'' \|.$  It follows that,
\begin{align}
\vert \max\limits_{p_i \in \Delta_i} U(p_i,q_{-i}(t)) - \max\limits_{p_i \in \Delta_i} U(p_i,r_{-i}(t))\vert   \leq K\| q_{-i}(t) - r_{-i}(t) \|,
\end{align}
implying the desired result.
\end{proof}
\vspace{-1em}
\begin{lemma}
Suppose $\vert a_t - b_t \vert = O(\frac{\log{t}}{t^r}),r>0$, $b_t \geq 0$ and $\sum_{t=1}^T \frac{a_t}{t} < \overline{B}$ is bounded above by $\overline{B} \in \mathbb{R}$ for all $T>0$ .  Then $\sum_{t=1}^T \frac{b_t}{t}$ converges as $T \rightarrow \infty$.
\label{IR8}
\end{lemma}
\begin{proof}
Let{\small
\begin{equation}
\delta_t :=
\begin{cases}
b_t - a_t & \mbox{if } b_t > a_t\\
0 & \mbox{otherwise.}
\end{cases}
\end{equation}
}
It follows that $\delta_t \geq 0$ and $b_t \leq a_t + \delta_t$. By hypothesis, $ \vert a_t - b_t \vert = O(\frac{\log{t}}{t^r})$, which implies that $\delta_t = O(\frac{\log{t}}{t^r})$. It follows that,
{\small
\begin{align}
\sum\limits_{t=1}^T \frac{b_t}{t} & \leq \sum\limits_{t=1}^T \frac{a_t + \delta_t}{t} = \sum\limits_{t=1}^T \frac{a_t}{t} + \sum\limits_{t=1}^T \frac{\delta_t}{t}.
\end{align}}
Since $\sum\limits_{t=1}^\infty \frac{a_t}{t} < \overline{B}$ is bounded above, $\sum\limits_{t=1}^\infty \frac{\delta_t}{t} < \infty$ converges, and $b_t\geq 0$, it follows that $\sum\limits_{t=1}^T \frac{b_t}{t} < \infty$ converges as $T \rightarrow \infty$.
\end{proof}
\vspace{-1em}
\begin{lemma}
Let $a_t = \sum\limits_{i=1}^n \left[v_i^m(\bar q(t)) - U(\bar q^n(t))\right]$, then $\lim\limits_{T\rightarrow \infty} \frac{a_1 + \ldots + a_T}{T} = 0$ implies that, for every $\varepsilon > 0$, $$\lim\limits_{T\rightarrow \infty} \frac{\#\{1\leq t \leq T: \bar q^n(t) \notin C_\varepsilon \}}{T} = 0.$$
\label{IR4}
\end{lemma}
\begin{IEEEproof}
Let $\varepsilon > 0$ be given.  By definition,
\begin{equation}
\bar q^n(t) \in C_\varepsilon \Leftrightarrow v_i^m(\bar q(t)) - U(\bar q^n(t)) < \varepsilon ~\forall i.
\label{IR4_eq1}
\end{equation}
The utility function $U(\cdot)$ is assumed to be permutation invariant for all players, so an equivalent statement to \eqref{IR4_eq1} is,
\begin{equation}
\bar q^n(t) \in C_\varepsilon \Leftrightarrow \sum\limits_{i=1}^n \left[v_i^m(\bar q(t)) - U(\bar q^n(t))\right] < n\varepsilon.
\end{equation}
Let
$$
b_t =
\begin{cases}
1, & \mbox{ if $a_t \geq n\varepsilon$}\\
0, & \mbox{ otherwise.}
\end{cases}
$$
Note that $b_t = 0 \Leftrightarrow \bar q^n(t) \in C_\varepsilon$ and $b_t = 1 \Leftrightarrow \bar q^n(t) \notin C_\varepsilon$, thus
\begin{equation}
\frac{\#\{1\leq t \leq T: \bar q^n(t) \notin C_\varepsilon \}}{T} = \frac{b_1+\ldots +b_T}{T}.
\end{equation}
Note also that $a_t \geq 0$. Clearly, $$ \frac{b_1 + \ldots + b_T}{T} \leq \frac{1}{n\varepsilon}\frac {a_1 + \ldots + a_T}{T},$$
implying $\lim\limits_{T \rightarrow \infty} \frac{b_1 + \ldots + b_T}{T} = 0$,
from which the desired result follows.
\end{IEEEproof}
\begin{lemma}
$\lim\limits_{T\rightarrow \infty} \frac{\#\{1\leq t \leq T :\bar q^n(t) \notin C_\varepsilon  \}}{T} = 0$ for all $\varepsilon > 0$ implies that $\lim\limits_{T\rightarrow \infty} \frac{\#\{1\leq t \leq T :\bar q^n(t) \notin B_\delta(C)  \}}{T} = 0$ for all $\delta > 0$.
\label{IR7}
\end{lemma}
\begin{IEEEproof}
Suppose $\lim\limits_{T\rightarrow \infty} \frac{\#\{1\leq t \leq T :\bar q^n(t) \notin C_\varepsilon  \}}{T} = 0$
for all $\varepsilon > 0$, but there exists some $\delta > 0$ such that $$\limsup\limits_{T\rightarrow \infty} \frac{\#\{1\leq t \leq T :\bar q^n(t) \notin B_\delta(C)  \}}{T} = \alpha > 0.$$
Then there exists an $\varepsilon' > 0$ such that $$\bar q^n(t) \notin B_\delta(C) \Rightarrow \bar q^n(t) \notin C_{\varepsilon'},$$  which implies that
\begin{align}
 \#\{ 1\leq t \leq T : \bar q^n(t) \notin C_{\varepsilon'} \} \geq \#\{ 1\leq t \leq T : \bar q^n(t) \notin B_\delta(C) \}.
\end{align}
This implies that $$\limsup\limits_{T\rightarrow \infty} \frac{\#\{1\leq t \leq T :\bar q^n(t) \notin C_{\varepsilon'}  \}}{T} \geq \alpha$$ for some $\varepsilon'>0$, a contradiction.
\end{IEEEproof}

\begin{lemma}
$\lim\limits_{T\rightarrow\infty} \frac{\#\{1\leq t \leq T:\bar q^n(t) \notin B_\delta(C)\}}{T}=0$ for all $\delta > 0$ implies  $\lim\limits_{t \rightarrow \infty} d(\bar q^n(t),C) = 0$.
\label{IR3}
\end{lemma}
The proof of this result closely follows the proof of \cite{Mond01}, Lemma 1, and is omitted here for brevity.
}
\bibliographystyle{IEEEtran}
\bibliography{myRefs}

\end{document}